\documentclass[a4paper,oneside,10pt]{article}%
\usepackage{amsmath}
\usepackage{amsfonts}
\usepackage{amssymb}
\usepackage{graphicx}
\usepackage[square,numbers,sort&compress]{natbib}%
\providecommand{\U}[1]{\protect\rule{.1in}{.1in}}
\newtheorem{theorem}{Theorem}[section]

\newtheorem{definition}[theorem]{Definition}

\newtheorem{lemma}[theorem]{Lemma}

\newtheorem{remark}[theorem]{Remark}

\newenvironment{proof}[1][Proof]{\noindent\textbf{#1.} }{\  $\Box$}
\numberwithin{equation}{section}

\begin{document}

\title{$G$-capacity under degenerate case and its application}
\author{Xiaojuan Li\thanks{Zhongtai Securities Institute for Financial Studies,
Shandong University, Jinan 250100, China. Email: lixiaojuan@mail.sdu.edu.cn. }
\and Xinpeng Li\thanks{Research Center for Mathematics and Interdisciplinary Sciences, Shandong University,
Qingdao 266237, China, lixinpeng@sdu.edu.cn. } }
\maketitle

\textbf{Abstract}. In this paper, we first find a type of viscosity solution
of $G$-heat equation under degenerate case, and then obtain the related
$G$-capacity $c(\{B_{T}\in A\})$ for any Borel set $A$. Furthermore, we prove
that $I_{A}(B_{T})$ has no quasi-continuous version when it is not a constant function.

{\textbf{Key words}. } $G$-heat equation, $G$-expectation, $G$-capacity,
Quasi-continuous, Viscosity solution

\textbf{AMS subject classifications.} 60H10

\addcontentsline{toc}{section}{\hspace*{1.8em}Abstract}

\section{Introduction}

Motivated by model uncertainty in finance, Peng \cite{P07a, P2019} introduced
the notions of $G$-expectation $\mathbb{\hat{E}}\left[  \cdot\right]  $ and
$G$-Brownian motion $(B_{t})_{t\geq0}$ via the following $G$-heat equation:%

\begin{equation}
\partial_{t}u-G\left(  \partial_{xx}^{2}u\right)  =0\text{, }u\left(
0,x\right)  =\varphi\left(  x\right)  , \label{1}%
\end{equation}
where $G\left(  a\right)  =\frac{1}{2}\left(  \overline{\sigma}^{2}%
a^{+}-\underline{\sigma}^{2}a^{-}\right)  $ for $a\in\mathbb{R}$,
$\overline{\sigma}>0$ and $\underline{\sigma}\in\lbrack0,\overline{\sigma}]$.
For any bounded and continuous function $\varphi$, we have $\mathbb{\hat{E}%
}\left[  \varphi\left(  x+B_{t}\right)  \right]  =u(t,x)$, where $u$ is the
viscosity solution of (\ref{1}). Under the $G$-expectation framework, the
corresponding stochastic calculus of It\^{o}'s type was also established in
Peng \cite{P07a,P08a}.

The $G$-expectation can be also seen as a upper expectation. Indeed, Denis et
al. \cite{DHP11} obtained a representation theorem of $G$-expectation
$\mathbb{\hat{E}}\left[  \cdot\right]  $ by stochastic control method:%
\[
\mathbb{\hat{E}}[X]=\sup_{P\in\mathcal{P}}E_{P}[X]\text{ for each }X\in
Lip(\Omega).
\]
where $\mathcal{P}$ is a family of weakly compact probability measures on
$(\Omega,\mathcal{B}(\Omega))$. Moreover, they gave the\ characterization of
$L_{G}^{p}(\Omega)$ for $p\geq1$. The reprensentation theorem was also
obtained in \cite{HP09} by a simple probabilistic method.

Denis et al. \cite{DHP11} (see also \cite{HP09}) introduced the notion of
$G$-capacity $c(\cdot)$ in $G$-expectation space and showed that each random
variable in $L_{G}^{1}(\Omega)$ has a quasi-continuous version with respect to
$c(\cdot)$. Under the non-degenerate case, i.e. $\underline{\sigma}>0$, Hu et
al. \cite{HWZ} proved that $c(\{B_{T}=a\})=0$ for each $(T,a)\in\left(
0,\infty\right)  \times\mathbb{R}$ by finding a kind of viscosity
supersolution of $G$-heat equation (\ref{1}), and further obtained that
$I_{[a,b]}(B_{T})$, $a\leq b$, is in $L_{G}^{1}(\Omega)$, which has important
application in stochastic recursive optimal control problem under
$G$-expectation space (see \cite{HJ1}). As far as we know, there is no result
about the above two points under the degenerate case, i.e. $\underline{\sigma
}=0$.

In this paper, we first study $c(\{B_{T}\in A\})$ for $A\in\mathcal{B}%
(\mathbb{R})$ under degenerate case. The known method for calculating
$G$-capacity is to find the \textquotedblleft similarity solution" of $G$-heat
equation (\ref{1}) (see \cite{HS, PWX, PYY}). But this method is no longer
suitable for some cases such as $A=(-\infty,a]\cup\lbrack b,\infty)$ with
$a<b$. To overcome this difficult, we use stochastic control method to find a
type of viscosity solution of $G$-heat equation (\ref{1}), and then obtain
$c(\{B_{T}\in A\})$, which can provide some examples for checking the
convergence rate of Peng's central limit theorem (see \cite{HL, K2}) under
$\underline{\sigma}=0$. Furthermore, we prove that $I_{A}(B_{T})$ is not in
$L_{G}^{1}(\Omega)$ for any $A\in\mathcal{B}(\mathbb{R})$ with $A\not =%
\emptyset$ and $A\not =\mathbb{R}$, which is completely different from the
case $\underline{\sigma}>0$.

This paper is organized as follows. In Section 2, we present some basic
notions and results of $G$-expectation. In Section 3, we obtain the
$G$-capacity $c(\{B_{T}\in A\})$ for any Borel set $A$ under degenerate case.
As an application, we prove that $I_{A}(B_{T})$ is not in $L_{G}^{1}(\Omega)$
for any $A\in\mathcal{B}(\mathbb{R})$ with $A\not =\emptyset$ and
$A\not =\mathbb{R}$ in Section 4.

\section{Preliminaries}

We recall some basic notions and results of $G$-expectation. The readers may
refer to \cite{P07a, P08a, P08b, P2019} for more details.

Let $\Omega=C\left[  0,\infty\right)  $ be the space of real-valued continuous
functions on $[0,\infty)$ with $\omega_{0}=0$. Let $B_{t}(\omega):=\omega_{t}%
$, for $\omega\in\Omega$ and $t\geqslant0$ be the canonical process . Set%
\[
Lip\left(  \Omega\right)  :=\left\{  \varphi\left(  B_{t_{1}},B_{t_{2}}%
,\cdots,B_{t_{n}}\right)  :n\in\mathbb{N},\text{ }0<t_{1}<\cdots<t_{n},\text{
}\varphi\in C_{b,Lip}\left(  \mathbb{R}^{n}\right)  \right\}  ,
\]
where $C_{b,Lip}\left(  \mathbb{R}^{n}\right)  $ denotes the space of bounded
Lipschitz functions on $\mathbb{R}^{n}$. It is easy to verify that%
\[
Lip\left(  \Omega\right)  =\left\{  \varphi\left(  B_{t_{1}},B_{t_{2}%
}-B_{t_{1}},\cdots,B_{t_{n}}-B_{t_{n-1}}\right)  :n\in\mathbb{N},\text{
}0<t_{1}<\cdots<t_{n},\text{ }\varphi\in C_{b,Lip}\left(  \mathbb{R}%
^{n}\right)  \right\}  .
\]

Let%
\[
G\left(  a\right)  :=\frac{1}{2}\left(  \overline{\sigma}^{2}a^{+}%
-\underline{\sigma}^{2}a^{-}\right)  \text{ for }a\in\mathbb{R},
\]
where $\overline{\sigma}>0$ and $\underline{\sigma}\in\lbrack0,\overline
{\sigma}]$. The $G$-expectation $\mathbb{\hat{E}}:$ $Lip\left(  \Omega\right)
\rightarrow\mathbb{R}$ is defined by the following two steps.

Step 1. For each $X=\varphi\left(  B_{t}-B_{s}\right)  $ with $0\leq s\leq t$
and $\varphi\in C_{b,Lip}\left(  \mathbb{R}\right)  $, we define%
\[
\mathbb{\hat{E}}\left[  X\right]  =u(t-s,0),
\]
where $u$ is the viscosity solution of (\ref{1}).

Step 2. For each $X=\varphi\left(  B_{t_{1}},B_{t_{2}}-B_{t_{1}}%
,\cdots,B_{t_{n}}-B_{t_{n-1}}\right)  $ with $0<t_{1}<\cdots<t_{n}$ and
$\varphi\in C_{b,Lip}\left(  \mathbb{R}^{n}\right)  $, we define%
\[
\mathbb{\hat{E}}\left[  X\right]  =\varphi_{0},
\]
where $\varphi_{0}$ is obtained via the following procedure:%
\[%
\begin{array}
[c]{rcl}%
\varphi_{n-1}(x_{1},\cdots,x_{n-1}) & = & \mathbb{\hat{E}}\left[
\varphi\left(  x_{1},\cdots,x_{n-1},B_{t_{n}}-B_{t_{n-1}}\right)  \right]  ,\\
\varphi_{n-2}(x_{1},\cdots,x_{n-2}) & = & \mathbb{\hat{E}}\left[
\varphi_{n-1}\left(  x_{1},\cdots,x_{n-2},B_{t_{n-1}}-B_{t_{n-2}}\right)
\right]  ,\\
& \vdots & \\
\varphi_{1}(x_{1}) & = & \mathbb{\hat{E}}\left[  \varphi_{2}\left(
x_{1},B_{t_{2}}-B_{t_{1}}\right)  \right]  ,\\
\varphi_{0} & = & \mathbb{\hat{E}}\left[  \varphi_{1}\left(  B_{t_{1}}\right)
\right]  .
\end{array}
\]
The following is the definition of the viscosity solution of (\ref{1}) (see
\cite{CIP}).

\begin{definition}
\label{d1}A real-valued continuous function $u\in C\left(  [0,\infty
)\times\mathbb{R}\right)  $ is called a viscosity subsolution (resp.
supersolution) of (\ref{1}) on $[0,\infty)\times\mathbb{R}$ if $u(0,\cdot
)\leq\varphi\left(  \cdot\right)  $ (resp. $u(0,\cdot)\geq\varphi\left(
\cdot\right)  $), and for all $\left(  t^{\ast},x^{\ast}\right)  \in\left(
0,\infty\right)  \times\mathbb{R}$, $\phi\in C^{2}\left(  \left(
0,\infty\right)  \times\mathbb{R}\right)  $ such that $u\left(  t^{\ast
},x^{\ast}\right)  =\phi\left(  t^{\ast},x^{\ast}\right)  $ and $u<\phi$
(resp. $u>\phi$) on $\left(  0,\infty\right)  \times\mathbb{R}\backslash
\left(  t^{\ast},x^{\ast}\right)  $, we have
\[
\partial_{t}\phi\left(  t^{\ast},x^{\ast}\right)  -G\left(  \partial_{xx}%
^{2}\phi\left(  t^{\ast},x^{\ast}\right)  \right)  \leq0\text{ (resp. }%
\geq0\text{)}.
\]

A real-valued continuous function $u\in C\left(  [0,\infty)\times
\mathbb{R}\right)  $ is called a viscosity solution of (\ref{1}) if it is both
a viscosity subsolution and a viscosity supersolution of (\ref{1}) on
$[0,\infty)\times\mathbb{R}$.
\end{definition}

The space $(\Omega,Lip\left(  \Omega\right)  ,\mathbb{\hat{E}})$ is called a
$G$-expectation space. The corresponding canonical process $(B_{t})_{t\geq0}$
is called a $G$-Brownian motion. The $G$-expectation $\mathbb{\hat{E}}:$
$Lip\left(  \Omega\right)  \rightarrow\mathbb{R}$ satisfies the following
properties: for each $X$, $Y\in$ $Lip\left(  \Omega\right)  $,

\begin{description}
\item[(i)] Monotonicity: If $X\geq Y$, then $\mathbb{\hat{E}}\left[  X\right]
\geq\mathbb{\hat{E}}\left[  Y\right]  .$

\item[(ii)] Constant preservation: $\mathbb{\hat{E}}\left[  c\right]  =c$ for
$c\in\mathbb{R}$.

\item[(iii)] Subadditivity: $\mathbb{\hat{E}}\left[  X+Y\right]
\leq\mathbb{\hat{E}}\left[  X\right]  +\mathbb{\hat{E}}\left[  Y\right]  $.

\item[(iv)] Positive homogeneity: $\mathbb{\hat{E}}\left[  \lambda X\right]
=\lambda\mathbb{\hat{E}}\left[  X\right]  $ for $\lambda\geq0$.
\end{description}

For every $p\geq1$, we denote by $L_{G}^{P}\left(  \Omega\right)  $ the
completion of $Lip\left(  \Omega\right)  $ under the norm $\left\Vert
X\right\Vert _{p}:=\left(  \mathbb{\hat{E}}\left[  \left\vert X\right\vert
^{p}\right]  \right)  ^{1/p}$. The $G$-expectation $\mathbb{\hat{E}}\left[
X\right]  $ can be extended continuously to $L_{G}^{1}\left(  \Omega\right)  $
under the norm $||\cdot||_{1}$, and $\mathbb{\hat{E}}:$ $L_{G}^{1}\left(
\Omega\right)  \rightarrow\mathbb{R}$ still satisfies (i)-(iv).

Denis et al. \cite{DHP11} (see also \cite{HP09}) proved the following
reprensentation theorem.

\begin{theorem}
There exists a weakly compact set of probability measures $\mathcal{P}$ on
$(\Omega,\mathcal{B}(\Omega))$ such that%
\[
\mathbb{\hat{E}}[X]=\sup_{P\in\mathcal{P}}E_{P}[X]\text{ for each }X\in
L_{G}^{1}(\Omega),
\]
where $\mathcal{B}(\Omega)=\sigma(B_{t}:t\geq0)$. $\mathcal{P}$ is called a
set that represents $\mathbb{\hat{E}}$.
\end{theorem}

\begin{remark}
\label{new-re-1}Denis et al. \cite{DHP11} gave a concrete $\mathcal{P}$ that
represents $\mathbb{\hat{E}}$ as follows. Let $(W_{t})_{t\geq0}$ be a
$1$-dimensional classical Brownian motion defined on a Wiener probability
space $(\tilde{\Omega},\mathcal{\tilde{F}},P^{W})$, and let $(\mathcal{\tilde
{F}}_{t})_{t\geq0}$ be the natural filtration generated by $W$. The set of
probability measures $\mathcal{P}_{1}$ on $(\Omega,\mathcal{B}(\Omega))$ is
defined by%
\[
\mathcal{P}_{1}=\left\{  P=P^{W}\circ(X^{v})^{-1}:X_{t}^{v}=\int_{0}^{t}%
v_{s}dW_{s},\text{ }(v_{s})_{s\leq T}\in M^{2}(0,T;[\underline{\sigma
},\overline{\sigma}])\text{ for any }T>0\right\}  ,
\]
where $M^{2}(0,T;[\underline{\sigma},\overline{\sigma}])$ is the space of all
$\mathcal{\tilde{F}}_{t}$-adapted processes $(v_{s})_{s\leq T}$ with $v_{s}%
\in\lbrack\underline{\sigma},\overline{\sigma}]$. Then $\mathcal{P}%
=\mathcal{\bar{P}}_{1}$ represents $\mathbb{\hat{E}}$, where $\mathcal{\bar
{P}}_{1}$ is the closure of $\mathcal{P}_{1}$ under the topology of weak convergence.
\end{remark}

The $G$-capacity associated to $\mathcal{P}$\ is defined as%
\begin{equation}
c(D)=\sup_{P\in\mathcal{P}}P(D)\text{ for }D\in\mathcal{B}(\Omega). \label{e2}%
\end{equation}
An important property of this capacity is that $c\left(  F_{n}\right)
\downarrow c\left(  F\right)  $ for any closed sets $F_{n}\downarrow F$.

A set $A\subset\mathcal{B}(\Omega)$ is polar if $c\left(  A\right)  =0$. A
property holds "quasi-surely" (q.s.) if it holds outside a polar set. In the
following, we do not distinguish two random variables $X$ and $Y$ if $X=Y$
q.s. For this $\mathcal{P}$, set%
\[
\mathbb{L}^{p}(\Omega):=\left\{  X\in\mathcal{B}(\Omega):\sup_{P\in
\mathcal{P}}E_{P}[|X|^{p}]<\infty\right\}  \text{ for }p\geq1\text{.}%
\]
It is easy to check that $L_{G}^{P}\left(  \Omega\right)  \subset$
$\mathbb{L}^{p}(\Omega)$. For each $X\in\mathbb{L}^{1}(\Omega)$,%
\[
\mathbb{\hat{E}}[X]:=\sup_{P\in\mathcal{P}}E_{P}[X]
\]
is still called $G$-expectation and satisfies (i)-(iv).

Now we review the\ characterization of $L_{G}^{p}(\Omega)$ for $p\geq1$.

\begin{definition}
A function $X:\Omega\rightarrow\mathbb{R}$ is said to be quasi-continuous if
for each $\varepsilon>0$, there exists an open set $O\subset\Omega$ with
$c(O)<\varepsilon$ such that $X|_{O^{c}}$ is continuous.
\end{definition}

\begin{definition}
We say that $X:\Omega\rightarrow\mathbb{R}$ has a quasi-continuous version if
there exists a quasi-continuous function $Y:\Omega\rightarrow\mathbb{R}$ such
that $X=Y$, q.s.
\end{definition}

\begin{theorem}
(\cite{DHP11, HP09})For each $p\geq1$, we have%
\[
L_{G}^{P}\left(  \Omega\right)  =\left\{  X\in\mathcal{B}(\Omega
):\lim_{N\rightarrow\infty}\mathbb{\hat{E}}\left[  |X|^{p}I_{\{|X|\geq
N\}}\right]  =0\text{ and }X\text{ has a quasi-continuous version}\right\}  .
\]

\end{theorem}

The following Fatou's property is important in $G$-expectation space.

\begin{theorem}
\label{new-31}(\cite{DHP11, HP09}) Let $\left\{  X_{n}\right\}  _{n=1}%
^{\infty}\subset L_{G}^{1}(\Omega)$ satisfy $X_{n}\downarrow X$ q.s. Then
\begin{equation}
\mathbb{\hat{E}}\left[  X_{n}\right]  \downarrow\mathbb{\hat{E}}\left[
X\right]  . \label{newnew4}%
\end{equation}
Moreover, if $X\in L_{G}^{1}(\Omega)$, then%
\[
\mathbb{\hat{E}}\left[  X_{n}-X\right]  \downarrow0.
\]

\end{theorem}

\section{Main results}

In this section, we consider the case $\underline{\sigma}=0$, then the
$G$-heat equation is%
\begin{equation}
\partial_{t}u-\frac{1}{2}\overline{\sigma}^{2}\left(  \partial_{xx}%
^{2}u\right)  ^{+}=0\text{, }u\left(  0,x\right)  =\varphi\left(  x\right)
.\label{e3}%
\end{equation}
The following theorem is our main result.

\begin{theorem}
\label{th1}Let $\underline{\sigma}=0$ and $\overline{\sigma}>0$. For each
given $T>0$ and $A\in\mathcal{B}(\mathbb{R})$, we have

\noindent$(\mathrm{i})$ If $\rho(A):=\inf\{|x|:x\in A\}=0$, then $c(\{B_{T}\in
A\})=1$;

\noindent$(\mathrm{ii})$ If $A\subset\lbrack0,\infty)$ or $A\subset
(-\infty,0]$, then $c(\{B_{T}\in A\})=\Phi\left(  \frac{\rho(A)}%
{\overline{\sigma}\sqrt{T}}\right)  $, where $\Phi(x)=\frac{2}{\sqrt{2\pi}%
}\int_{x}^{\infty}\exp\left(  -\frac{r^{2}}{2}\right)  dr$;

\noindent$(\mathrm{iii})$ If $\rho(A)\not =0$, $A\varsubsetneq\lbrack
0,\infty)$ and $A\varsubsetneq(-\infty,0]$, then%
\[
c(\{B_{T}\in A\})=\sum_{i=-\infty}^{\infty}sgn(i)\left[  \Phi\left(
\frac{|2i(\rho(A^{+})+\rho(A^{-}))+\rho(A^{-})|}{\overline{\sigma}\sqrt{T}%
}\right)  +\Phi\left(  \frac{|2i(\rho(A^{+})+\rho(A^{-}))+\rho(A^{+}%
)|}{\overline{\sigma}\sqrt{T}}\right)  \right]  ,
\]
where $sgn(x):=I_{[0,\infty)}(x)-I_{(-\infty,0)}(x)$, $\rho(A^{+}%
):=\inf\{x:x\in A,$ $x\geq0\}$, $\rho(A^{-}):=\inf\{-x:x\in A,$ $x\leq0\}$.
\end{theorem}

In order to prove this theorem, we need the following lemma. By Remark
\ref{new-re-1}, we know that
\begin{equation}
\sup_{v\in M^{2}(0,T;[0,\overline{\sigma}])}P^{W}\left(  \left\{  \int_{0}%
^{T}v_{s}dW_{s}\in A\right\}  \right)  \leq c(\{B_{T}\in A\})\text{ for }%
A\in\mathcal{B}(\mathbb{R}), \label{e4}%
\end{equation}
where $(W_{t})_{t\geq0}$ is a $1$-dimensional classical Brownian motion
defined on a Wiener probability space $(\tilde{\Omega},\mathcal{\tilde{F}%
},P^{W})$. The proof of the following lemma is accomplished by the proper
construction of $u_{n}$ on the assumption that the equal sign in (\ref{e4})
holds and the optimal control in (\ref{e4}) is $\left(  \overline{\sigma
}I_{[0,\tau_{b}\wedge\tau_{l}]}(s)\right)  _{s\leq T}$ for $A=\{b,l\}$ with
$b\leq0\leq l$, where%
\[
\tau_{a}:=\inf\{t\geq0:\overline{\sigma}W_{t}=a\}\text{ for }a\in\mathbb{R}.
\]
In turn, the conclusion of the lemma also shows that the above assumption is true.

\begin{lemma}
\label{le2}Let $\underline{\sigma}=0$ and $\overline{\sigma}>0$. Then, for
each given $T>0$ and $b<0<l$, we have%
\[
c(\{B_{T}\in\{b,l\}\})=\sum_{i=-\infty}^{\infty}sgn(i)\left[  \Phi\left(
\frac{|2i(l-b)-b|}{\overline{\sigma}\sqrt{T}}\right)  +\Phi\left(
\frac{|2i(l-b)+l|}{\overline{\sigma}\sqrt{T}}\right)  \right]  .
\]

\end{lemma}

\begin{proof}
For each fixed $n\geq1$, we first prove that
\begin{align*}
u_{n}(t,x)  &  :=\sum_{i=-\infty}^{\infty}sgn(i)\left[  \Phi\left(
\frac{|2i(l-b)+x-b|}{\overline{\sigma}\sqrt{\frac{1}{n}+t}}\right)
+\Phi\left(  \frac{|2i(l-b)+l-x|}{\overline{\sigma}\sqrt{\frac{1}{n}+t}%
}\right)  \right]  I_{(b,l)}(x)\\
&  \ \ \ \ \ +\Phi\left(  \frac{|b-x|\wedge|l-x|}{\overline{\sigma}\sqrt
{\frac{1}{n}+t}}\right)  I_{(-\infty,b]\cup\lbrack l,\infty)}(x)
\end{align*}
is a viscosity solution of (\ref{e3}) on $(0,\infty)\times\mathbb{R}$.

On the one hand, we need to prove that $u_{n}$ is a viscosity subsolution of
(\ref{e3}). Indeed, it suffices to do three steps to show that $\partial
_{t}\psi\left(  t^{\ast},x^{\ast}\right)  -\frac{1}{2}\overline{\sigma}%
^{2}\left(  \partial_{xx}^{2}\psi\left(  t^{\ast},x^{\ast}\right)  \right)
^{+}\leq0$ for all $\left(  t^{\ast},x^{\ast}\right)  \in\left(
0,\infty\right)  \times\mathbb{R}$, $\psi\in C^{2}\left(  \left(
0,\infty\right)  \times\mathbb{R}\right)  $ such that $\psi\left(  t^{\ast
},x^{\ast}\right)  =u_{n}\left(  t^{\ast},x^{\ast}\right)  $ and $\psi\geq
u_{n}$.

Step 1. If $x^{\ast}\in(-\infty,b)\cup(l,\infty)$, it is easy to know from
extreme value theory that
\[
\partial_{xx}^{2}\psi\left(  t^{\ast},x^{\ast}\right)  \geq\partial_{xx}%
^{2}u_{n}\left(  t^{\ast},x^{\ast}\right)  ,\text{ \ }\partial_{t}\psi\left(
t^{\ast},x^{\ast}\right)  =\partial_{t}u_{n}\left(  t^{\ast},x^{\ast}\right)
\]
and it follows by simple calculation that%
\[
\partial_{t}u_{n}\left(  t^{\ast},x^{\ast}\right)  =-\frac{1}{2}\Phi^{\prime
}\left(  \frac{|b-x^{\ast}|\wedge|l-x^{\ast}|}{\overline{\sigma}\sqrt{\frac
{1}{n}+t^{\ast}}}\right)  \frac{|b-x^{\ast}|\wedge|l-x^{\ast}|}{\overline
{\sigma}\sqrt{\left(  \frac{1}{n}+t^{\ast}\right)  ^{3}}}%
\]
and%
\[
\partial_{xx}^{2}u_{n}\left(  t^{\ast},x^{\ast}\right)  =\Phi^{\prime\prime
}\left(  \frac{|b-x^{\ast}|\wedge|l-x^{\ast}|}{\overline{\sigma}\sqrt{\frac
{1}{n}+t^{\ast}}}\right)  \frac{1}{\overline{\sigma}^{2}\left(  \frac{1}%
{n}+t^{\ast}\right)  }.
\]
By the definition of $\Phi\left(  x\right)  $, we know that $\Phi
^{^{\prime\prime}}\left(  x\right)  >0$ holds if $x>0$, and
\[
\Phi^{^{\prime\prime}}\left(  x\right)  =-x\Phi^{^{\prime}}\left(  x\right)
.
\]
Then we obtain
\[
\partial_{xx}^{2}u_{n}\left(  t^{\ast},x^{\ast}\right)  >0\text{ and }%
\partial_{t}u_{n}\left(  t^{\ast},x^{\ast}\right)  -\frac{1}{2}\overline
{\sigma}^{2}\partial_{xx}^{2}u_{n}\left(  t^{\ast},x^{\ast}\right)  =0,
\]
which implies $\partial_{t}\psi\left(  t^{\ast},x^{\ast}\right)  -\frac{1}%
{2}\overline{\sigma}^{2}\left(  \partial_{xx}^{2}\psi\left(  t^{\ast},x^{\ast
}\right)  \right)  ^{+}\leq0$.

Step 2. If $x^{\ast}\in(b,l)$, by simple calculation, we can still get%
\begin{align*}
\partial_{t}u_{n}\left(  t^{\ast},x^{\ast}\right)   &  =-\frac{1}{2}%
\sum_{i=-\infty}^{\infty}sgn(i)\left[  \Phi^{^{\prime}}\left(  \frac
{|2i(l-b)+x^{\ast}-b|}{\overline{\sigma}\sqrt{\frac{1}{n}+t^{\ast}}}\right)
\frac{|2i(l-b)+x^{\ast}-b|}{\overline{\sigma}\sqrt{\left(  \frac{1}{n}%
+t^{\ast}\right)  ^{3}}}\right. \\
&  \left.  \text{ \ \ }+\Phi^{^{\prime}}\left(  \frac{|2i(l-b)+l-x^{\ast}%
|}{\overline{\sigma}\sqrt{\frac{1}{n}+t^{\ast}}}\right)  \frac
{|2i(l-b)+l-x^{\ast}|}{\overline{\sigma}\sqrt{\left(  \frac{1}{n}+t^{\ast
}\right)  ^{3}}}\right]
\end{align*}
and%
\[
\partial_{xx}^{2}u_{n}\left(  t^{\ast},x^{\ast}\right)  =\sum_{i=-\infty
}^{\infty}sgn(i)\left[  \Phi^{^{\prime\prime}}\left(  \frac{|2i(l-b)+x^{\ast
}-b|}{\overline{\sigma}\sqrt{\frac{1}{n}+t^{\ast}}}\right)  +\Phi
^{^{\prime\prime}}\left(  \frac{|2i(l-b)+l-x^{\ast}|}{\overline{\sigma}%
\sqrt{\frac{1}{n}+t^{\ast}}}\right)  \right]  \frac{1}{\overline{\sigma}%
^{2}\left(  \frac{1}{n}+t^{\ast}\right)  ^{2}}.
\]
It can be verified by the definition of $\Phi\left(  x\right)  $ that
$\partial_{t}u_{n}\left(  t^{\ast},x^{\ast}\right)  -\frac{1}{2}%
\overline{\sigma}^{2}\partial_{xx}^{2}u_{n}\left(  t^{\ast},x^{\ast}\right)
=0$. Since $\partial_{xx}^{2}\psi\left(  t^{\ast},x^{\ast}\right)
\geq\partial_{xx}^{2}u_{n}\left(  t^{\ast},x^{\ast}\right)  $, and if we want
to replicate the idea we had in step 1, we just have to prove $\partial
_{t}u_{n}\left(  t^{\ast},x^{\ast}\right)  >0$. It is well-known that the
stopping time $\tau_{b-x}\wedge\tau_{l-x}$ has the following density for
$x\in(b,l)$ (see \cite{KS})%
\begin{align*}
P^{W}\left(  \left\{  \tau_{b-x}\wedge\tau_{l-x}\in ds\right\}  \right)   &
=\frac{1}{\sqrt{2\pi\overline{\sigma}^{2}s^{3}}}\sum_{i=-\infty}^{\infty
}\left\{  (2i(l-b)-b+x)\exp\left(  -\frac{(2i(l-b)-b+x)^{2}}{2\overline
{\sigma}^{2}s}\right)  \right. \\
&  \ \ \ \left.  +(2i(l-b)+l-x)\exp\left(  -\frac{(2i(l-b)+l-x)^{2}%
}{2\overline{\sigma}^{2}s}\right)  \right\}  ds\text{.}%
\end{align*}
By tedious calculation, we can get%
\[
u_{n}\left(  t,x\right)  =\int_{0}^{\frac{1}{n}+t}P^{W}\left(  \left\{
\tau_{b-x}\wedge\tau_{l-x}\in ds\right\}  \right)  \text{ for }x\in(b,l),
\]
which implies $\partial_{t}u_{n}\left(  t^{\ast},x^{\ast}\right)  >0$, so, we
get $\partial_{xx}^{2}u_{n}\left(  t^{\ast},x^{\ast}\right)  >0$. Moreover,%
\[
\partial_{t}\psi\left(  t^{\ast},x^{\ast}\right)  -\frac{1}{2}\overline
{\sigma}^{2}\left(  \partial_{xx}^{2}\psi\left(  t^{\ast},x^{\ast}\right)
\right)  ^{+}\leq0.
\]

Step 3. If $x^{\ast}\in\{b,l\}$, we know by the definition of $\psi$ that
\[
\psi\left(  t,x^{\ast}\right)  \geq u_{n}\left(  t,x^{\ast}\right)  =1\text{
and }\psi\left(  t^{\ast},x^{\ast}\right)  =u_{n}\left(  t^{\ast},x^{\ast
}\right)  =1,
\]
it then follows that $\partial_{t}\psi\left(  t^{\ast},x^{\ast}\right)  =0$.
Hence, we obtain $\partial_{t}\psi\left(  t^{\ast},x^{\ast}\right)  -\frac
{1}{2}\overline{\sigma}^{2}\left(  \partial_{xx}^{2}\psi\left(  t^{\ast
},x^{\ast}\right)  \right)  ^{+}\leq0$. Thus, from step 1 to step 3, we know
that $u_{n}$ is a viscosity subsolution of (\ref{e3}).

On the other hand, we need to prove that $u_{n}$ is a viscosity supersolution
of (\ref{e3}). So, let us show in two steps that $\partial_{t}\psi\left(
t^{\ast},x^{\ast}\right)  -\frac{1}{2}\overline{\sigma}^{2}\left(
\partial_{xx}^{2}\psi\left(  t^{\ast},x^{\ast}\right)  \right)  ^{+}\geq0$ for
all $\left(  t^{\ast},x^{\ast}\right)  \in\left(  0,\infty\right)
\times\mathbb{R}$, $\psi\in C^{2}\left(  \left(  0,\infty\right)
\times\mathbb{R}\right)  $ such that $\psi\left(  t^{\ast},x^{\ast}\right)
=u_{n}\left(  t^{\ast},x^{\ast}\right)  $ and $\psi\leq u_{n}$.

Step 4. If $x^{\ast}\not \in \{b,l\}$, similar to the proof idea of step 1 and
step 2, we can also get the same results as follows:%
\[
\partial_{xx}^{2}\psi\left(  t^{\ast},x^{\ast}\right)  \leq\partial_{xx}%
^{2}u_{n}\left(  t^{\ast},x^{\ast}\right)  ,\text{ \ }\partial_{xx}^{2}%
u_{n}\left(  t^{\ast},x^{\ast}\right)  >0,\text{ \ }\partial_{t}\psi\left(
t^{\ast},x^{\ast}\right)  =\partial_{t}u_{n}\left(  t^{\ast},x^{\ast}\right)
.
\]
and
\[
\partial_{t}u_{n}\left(  t^{\ast},x^{\ast}\right)  -\frac{1}{2}\overline
{\sigma}^{2}\partial_{xx}^{2}u_{n}\left(  t^{\ast},x^{\ast}\right)  =0
\]
Thereby, we have $\left(  \partial_{xx}^{2}\psi\left(  t^{\ast},x^{\ast
}\right)  \right)  ^{+}\leq\left(  \partial_{xx}^{2}u_{n}\left(  t^{\ast
},x^{\ast}\right)  \right)  ^{+}$. Moreover, we show that
\begin{align*}
&  \partial_{t}\psi\left(  t^{\ast},x^{\ast}\right)  -\frac{1}{2}%
\overline{\sigma}^{2}\left(  \partial_{xx}^{2}\psi\left(  t^{\ast},x^{\ast
}\right)  \right)  ^{+}\\
&  \geq\partial_{t}u_{n}\left(  t^{\ast},x^{\ast}\right)  -\frac{1}%
{2}\overline{\sigma}^{2}\left(  \partial_{xx}^{2}u_{n}\left(  t^{\ast}%
,x^{\ast}\right)  \right)  ^{+}\\
&  =\partial_{t}u_{n}\left(  t^{\ast},x^{\ast}\right)  -\frac{1}{2}%
\overline{\sigma}^{2}\partial_{xx}^{2}u_{n}\left(  t^{\ast},x^{\ast}\right) \\
&  =0
\end{align*}

Step 5. If $x^{\ast}\in\{b,l\}$, then we know that
\[
\psi\left(  t^{\ast},x^{\ast}\right)  =u_{n}\left(  t^{\ast},x^{\ast}\right)
\text{ and }\psi\left(  t^{\ast},x\right)  \leq u_{n}\left(  t^{\ast
},x\right)  ,
\]
which implies that%
\begin{equation}
\partial_{x+}\psi\left(  t^{\ast},x^{\ast}\right)  \leq\partial_{x+}%
u_{n}\left(  t^{\ast},x^{\ast}\right)  \text{ and }\partial_{x-}\psi\left(
t^{\ast},x^{\ast}\right)  \geq\partial_{x-}u_{n}\left(  t^{\ast},x^{\ast
}\right)  . \label{newe2}%
\end{equation}
It is easy to check by the definition of $\Phi\left(  x\right)  $ that
\begin{align*}
\partial_{x-}u_{n}\left(  t^{\ast},b\right)   &  =\frac{2}{\bar{\sigma}%
\sqrt{2\pi\left(  \frac{1}{n}+t^{\ast}\right)  }}>0,\\
\partial_{x+}u_{n}\left(  t^{\ast},l\right)   &  =-\frac{2}{\bar{\sigma}%
\sqrt{2\pi\left(  \frac{1}{n}+t^{\ast}\right)  }}<0.
\end{align*}
Now we claim that $\partial_{x+}u_{n}\left(  t^{\ast},b\right)  <0$.
Otherwise, $\partial_{x+}u_{n}\left(  t^{\ast},b\right)  \geq0$. Noting that
$\partial_{xx}^{2}u_{n}\left(  t^{\ast},x\right)  >0$ for $x\in(b,l)$ and
$u_{n}\left(  t^{\ast},b\right)  =1$, we obtain $u_{n}\left(  t^{\ast
},x\right)  >1$ for $x\in(b,l)$, which contradicts to $u_{n}\leq1$. Hence, we
have $\partial_{x+}u_{n}\left(  t^{\ast},b\right)  <0$. Similarly, we can get
$\partial_{x-}u_{n}\left(  t^{\ast},l\right)  >0$. So we can not find $\psi\in
C^{2}\left(  \left(  0,\infty\right)  \times\mathbb{R}\right)  $ satisfying
(\ref{newe2}). Thus, from (H4)-(H5), we know that $u_{n}$ is also a viscosity
supersolution of (\ref{e3}).

In conclusion, $u_{n}$ is a viscosity solution of (\ref{e3}). Then we obtain
\[
\mathbb{\hat{E}}\left[  u_{n}\left(  0,B_{T}\right)  \right]  =u_{n}(T,0).
\]
Noting that $u_{n}\left(  0,x\right)  \downarrow I_{\{b,l\}}(x)$, we can
deduce by Theorem \ref{new-31} that%
\[
c(\{B_{T}\in\{b,l\}\})=\mathbb{\hat{E}}\left[  I_{\{b,l\}}\left(
B_{T}\right)  \right]  =\lim_{n\rightarrow\infty}\mathbb{\hat{E}}\left[
u_{n}\left(  0,B_{T}\right)  \right]  =\lim_{n\rightarrow\infty}u_{n}(T,0),
\]
which implies the desired result.
\end{proof}

\begin{remark}
\label{re2}Under the condition of Lemma \ref{le2}, by using the similar
method, we can get
\[
c(\{B_{T}\in\{b,l\}\})=c(\{B_{T}\in(-\infty,b]\cup\lbrack l,\infty)\})\text{
for }b<0<l
\]
and%
\[
c(\{B_{T}=a\})=c(\{B_{T}\geq|a|\})=c(\{B_{T}\leq-|a|\})=\Phi\left(  \frac
{|a|}{\overline{\sigma}\sqrt{T}}\right)  \text{ for }a\in\mathbb{R}\text{.}%
\]
The value $c(\{B_{T}\geq|a|\})$ was obtained in \cite{PWX, PYY} by using
\textquotedblleft similarity solution" method.
\end{remark}

\textbf{Proof of Theorem \ref{th1}} If $\rho(A)=0$, we can find a sequence
$\{a_{n}:n\geq1\}\subset A$ such that $a_{n}\rightarrow0$. By Remark
\ref{re2}, we get%
\[
c(\{B_{T}\in A\})\geq\lim_{n\rightarrow\infty}c(\{B_{T}=a_{n}\})=\lim
_{n\rightarrow\infty}\Phi\left(  \frac{|a_{n}|}{\overline{\sigma}\sqrt{T}%
}\right)  =1,
\]
which implies (i).

If $A\subset\lbrack0,\infty)$, we can find a sequence $\{x_{n}:n\geq1\}\subset
A$ such that $x_{n}\rightarrow\rho(A)$. By Remark \ref{re2}, we know that%
\[
c(\{B_{T}\in A\})\geq\lim_{n\rightarrow\infty}c(\{B_{T}=x_{n}\})=\Phi\left(
\frac{\rho(A)}{\overline{\sigma}\sqrt{T}}\right)  .
\]
Noting that $A\subset\lbrack\rho(A),\infty)$, by Remark \ref{re2}, we get%
\[
c(\{B_{T}\in A\})\leq c(\{B_{T}\geq\rho(A)\})=\Phi\left(  \frac{\rho
(A)}{\overline{\sigma}\sqrt{T}}\right)  .
\]
Thus we have $c(\{B_{T}\in A\})=\Phi\left(  \frac{\rho(A)}{\overline{\sigma
}\sqrt{T}}\right)  $. By the similar method for $A\subset(-\infty,0]$, then we
obtain (ii).

If $\rho(A)\not =0$, $A\varsubsetneq\lbrack0,\infty)$ and $A\varsubsetneq
(-\infty,0]$, then we can find two sequences $\{b_{n}:n\geq1\}$ and
$\{c_{n}:n\geq1\}$ in $A$ such that $b_{n}<0<c_{n}$, $-b_{n}\rightarrow
\rho(A^{-})$ and $c_{n}\rightarrow\rho(A^{+})$. By Lemma \ref{le2}, Remark
\ref{re2} and $A\subset(-\infty,-\rho(A^{-})]\cup\lbrack\rho(A^{+}),\infty)$,
we get%
\[
c(\{B_{T}\in A\})\geq\lim_{n\rightarrow\infty}c(\{B_{T}\in\{b_{n}%
,c_{n}\}\})=c(\{B_{T}\in\{-\rho(A^{-}),\rho(A^{+})\}\}),
\]%
\[
c(\{B_{T}\in A\})\leq c(\{B_{T}\in(-\infty,-\rho(A^{-})]\cup\lbrack\rho
(A^{+}),\infty)\}),
\]
which implies $c(\{B_{T}\in A\})=c(\{B_{T}\in\{-\rho(A^{-}),\rho(A^{+})\}\})$.
Thus we obtain (iii). $\Box$

\section{Application to $G$-expectation}

For each $\phi_{n}\in C_{b}(\mathbb{R})$ such that $\phi_{n}\downarrow I_{A}$,
we know by (\ref{newnew4}) that%
\begin{equation}
\mathbb{\hat{E}}\left[  \phi_{n}\left(  B_{T}\right)  \right]  \downarrow
c(\{B_{T}\in A\}). \label{newnew5}%
\end{equation}
The following theorem is the application of Theorem \ref{th1}.

\begin{theorem}
Let $\underline{\sigma}=0$ and $\overline{\sigma}>0$. Then, for each given
$T>0$, $A\in\mathcal{B}(\mathbb{R})$ with $A\not =\emptyset$ and
$A\not =\mathbb{R}$, we have $I_{A}(B_{T})\notin L_{G}^{1}(\Omega)$.
\end{theorem}

\begin{proof}
Due to $A\not =\emptyset$ and $A\not =\mathbb{R}$, then one of the following
two results must hold.

\begin{description}
\item[(i)] There exist a point $x_{0}\in A$ and a sequence $\{x_{k}%
:k\geq1\}\subset A^{c}$ such that $x_{k}\rightarrow x_{0}$.

\item[(ii)] There exist a point $x_{0}\in A^{c}$ and a sequence $\{x_{k}%
:k\geq1\}\subset A$ such that $x_{k}\rightarrow x_{0}$.
\end{description}

If (i) holds and $I_{A}(B_{T})\in L_{G}^{1}(\Omega)$, then%
\[
h_{n}(B_{T})\vee I_{A}(B_{T})\in L_{G}^{1}(\Omega)\text{ and }h_{n}(B_{T})\vee
I_{A}(B_{T})-I_{A}(B_{T})\downarrow0,
\]
where
\[
h_{n}(x)=\left[  1+n(x-x_{0})\right]  I_{[x_{0}-\frac{1}{n},x_{0}]}(x)+\left[
1-n(x-x_{0})\right]  I_{(x_{0},x_{0}+\frac{1}{n}]}(x).
\]
By (\ref{newnew4}), we have%
\begin{equation}
\mathbb{\hat{E}}\left[  h_{n}(B_{T})\vee I_{A}(B_{T})-I_{A}(B_{T})\right]
\downarrow0\text{ as }n\rightarrow\infty. \label{newnew6}%
\end{equation}
Moreover, we also know that $h_{n}(B_{T})\vee I_{A}(B_{T})-I_{A}(B_{T})\geq
h_{n}(x_{k})I_{\{x_{k}\}}(B_{T})$ for any $k\geq1$. Then we deduce by Theorem
\ref{th1} that%
\[
\mathbb{\hat{E}}\left[  h_{n}(B_{T})\vee I_{A}(B_{T})-I_{A}(B_{T})\right]
\geq\lim_{k\rightarrow\infty}h_{n}(x_{k})c(\{B_{T}=x_{k}\})=c(\{B_{T}%
=x_{0}\})>0,
\]
which contradicts to (\ref{newnew6}). Thus $I_{A}(B_{T})\notin L_{G}%
^{1}(\Omega)$.

If (ii) holds, then $A^{c}$ satisfies (i). Thus we obtain $I_{A^{c}}%
(B_{T})\notin L_{G}^{1}(\Omega)$, which implies $I_{A}(B_{T})=1-I_{A^{c}%
}(B_{T})\notin L_{G}^{1}(\Omega)$.
\end{proof}

\section*{Acknowledgements}

The authors are supported by National Key R\&D Program of China (No. 2018YFA0703900) and National Natural Science Foundation of
China (No.  11671231 and 11601281).


\bigskip

\begin{thebibliography}{99}                                                                                               %


\bibitem {CIP}M.G. Crandall, H. Ishii, P.L. Lions, User's guide to viscosity
solutions of second order partial differential equations,\emph{ }Bull. Amer.
Math. Soc., 27 (1992), 1-67.

\bibitem {DHP11}L. Denis, M. Hu, S. Peng, Function spaces and capacity related
to a sublinear expectation: application to $G$-Brownian motion paths,
Potential Anal., 34 (2011), 139-161.

\bibitem {HJ1}M. Hu, S. Ji, Dynamic programming principle for stochastic
recursive optimal control problem driven by a G-Brownian motion, Stochastic
Process. Appl., 127 (2017), 107-134.

\bibitem {HP09}M. Hu, S. Peng, On representation theorem of G-expectations and
paths of $G$-Brownian motion, Acta Math. Appl. Sin. Engl. Ser., 25 (2009), 539-546.

\bibitem {HS}M. Hu, Y. Sun, Explicit positive solutions to $G$-heat equations
and the application to $G$-capacities, J. Differential Equations, 297 (2021), 246-276.

\bibitem {HWZ}M. Hu, F. Wang, G. Zheng, Quasi-continuous random variables and
processes under the $G$-expectation framework, Stochastic Process. Appl., 126
(2016), 2367-2387.

\bibitem {HL}S. Huang, G. Liang, A monotone scheme for $G$-equations with
application to the explicit convergence rate of robust central limit theorem,
arXiv:1904.07184v3, 2019.

\bibitem {KS}I. Karatzas, S.E. Shreve, Brownian Motion and Stochastic
Calculus, Springer (2005).

\bibitem {K2}N.V. Krylov, On Shige Peng's central limit theorem, Stochastic
Process. Appl., 130 (2020), 1426-1434.

\bibitem {PWX}Z. Pei, X. Wang, Y. Xu, X. Yue, A worst-case risk measure by
$G$-VaR, Acta Math. Appl. Sin. Engl. Ser., 37 (2021), 421-440.

\bibitem {P07a}S. Peng, $G$-expectation, $G$-Brownian Motion and Related
Stochastic Calculus of It\^{o} type, Stochastic analysis and applications,
Abel Symp., Vol. 2, Springer, Berlin, 2007, 541-567.

\bibitem {P08a}S. Peng, Multi-dimensional $G$-Brownian motion and related
stochastic calculus under $G$-expectation, Stochastic Process. Appl.
118(12)(2008) 2223-2253.

\bibitem {P08b}S. Peng, A new central limit theorem under sublinear
expectation, 2008. arXiv:0803.2656v1 [math.PR].

\bibitem {P2019}S. Peng, Nonlinear Expectations and Stochastic Calculus under
Uncertainty, Springer (2019).

\bibitem {PYY}S. Peng, S. Yang, J. Yao, Improving value-at-risk prediction
under model uncertainty, J. Financ. Econom., 2021, 1-32.
\end{thebibliography}
\end{document}